\documentclass{amsart}
\usepackage{amssymb, amsmath}
\usepackage{amssymb, amsmath, color}

\newtheorem{theorem}{Theorem}
\newtheorem{lemma}[theorem]{Lemma}

\newtheorem{remark}[theorem]{Remark}

\newtheorem{proposition}[theorem]{Proposition}

\newcommand{\algebraNomizu}{\mathfrak{S}}
\newcommand{\chr}{\Phi}
\newcommand{\VE}{X}
\newcommand{\ve}{x}
\newcommand{\Cottontensor}{C}

\begin{document}

\title[Homogeneous   Cotton solitons]
{Homogeneous   Cotton solitons}
\author{\quad E. Calvi\~{n}o-Louzao, L. M. Hervella, J. Seoane-Bascoy, R. V\'{a}zquez-Lorenzo}
\address{Department of Geometry and Topology, Faculty of Mathematics,
University of Santiago de Compostela, 15782 Santiago de Compostela, Spain}
\email{estebcl@edu.xunta.es, luismaria.hervella@usc.es,javier.seoane@usc.es,
ravazlor@edu.xunta.es}
\thanks{Supported by projects MTM2009-07756 and INCITE09 207 151 PR (Spain)}
\date{}

\begin{abstract} Left-invariant Cotton solitons on   homogeneous manifolds
are determined. Moreover,  algebraic Cotton solitons are studied providing  examples of non-invariant Cotton solitons,   both in the
Riemannian and Lorentzian  homogeneous settings.
\end{abstract}

\maketitle

\section{Introduction}\label{se:1}

The objective of the different geometric evolution equations  is to improve a given metric by
considering a   flow associated to the geometric object under consideration. The Ricci, Yamabe and
mean curvature flows are examples extensively studied in the literature. Under suitable conditions
the Ricci flow evolves an initial metric to an Einstein metric while the Yamabe flow evolves an
initial metric to a new one with constant scalar curvature within the same conformal class. Cotton, Yamabe and Ricci flows have many physical  applications (we refer to
\cite{Hall-CQG,Hall-Capocci99,Lashkari-Maloney10} and the references therein for more information). There
are however certain metrics which, instead of evolving by the flow, remain invariant. Such is the
case of those solitons   associated to self-similar solutions of the flow.

In the study of conformal geometry in dimensions greater than  three the Weyl tensor plays a
distinguished role, since its nullity characterizes local conformal flatness. The three-dimensional
case must be studied in a different way due to the fact that the Weyl tensor vanishes identically.
Moreover, the whole curvature is completely determined by the Ricci tensor, $\rho$. Local conformal
flatness is characterized in dimension three by the fact that the Schouten tensor, defined by
$S_{ij}=\rho_{ij}-\frac{\tau}{4}g_{ij}$ where $\tau$ denotes the scalar curvature,  is a Codazzi
tensor, or  equivalently  the Cotton tensor $\Cottontensor_{ijk}=(\nabla_i S)_{jk}-(\nabla_j
S)_{ik}$, which is the unique conformal invariant in dimension three, vanishes. We refer to
\cite{BJM, K, SFR} and the references therein for more information on the usefulness of the Cotton
tensor in describing the geometry of three-dimensional manifolds.

 The Cotton tensor appears  naturally in many physical contexts \cite{Ferreiro10,Garcia-Hehl-Heinicke04}, specially in Chern-Simons theory
\cite{Guven07} or topologically massive gravity \cite{Aliev-Nutku96,Chow-Pope-Sezgin10,Lashkari-Maloney10}. In particular, field equation in topologically massive gravity implies a proportionality between the Einstein and Cotton tensors. The fact that the Einstein tensor consists of second order derivatives on the metric whereas the Cotton tensor is of order three implies that an exact solution of this field equation is difficult to find in general. Indeed, most of the solutions for the field equation in topological massive gravity are constructed on homogeneous spaces.

In \cite{Ali-Ozgur08}  a new geometric flow  based on the conformally invariant Cotton tensor was
introduced. A Cotton flow is a one-parameter family $g(t)$ of three-dimensional metrics satisfying
\begin{equation}\label{Cotton flow}
    \frac{\partial}{\partial t}\,g(t)=-\lambda\, \Cottontensor_{g(t)},
\end{equation}
where $\Cottontensor_{g(t)}$ is the $(0,2)$-Cotton tensor  corresponding to the metric $g(t)$,
obtained from the $(0,3)$-Cotton tensor by means of the $\star$-Hodge operator, and given by
\[
\Cottontensor_{ij}=\frac{1}{2\sqrt{g}}\Cottontensor_{nmi}\epsilon^{nm\ell}g_{\ell j},
\]
where $\epsilon^{ijk}$ denotes   the Levi-Civita permutation symbol ($\epsilon^{123}=1$).

When comparing with the Yamabe flow, the Cotton flow has in some sense an opposite behaviour since
the Cotton flow changes the conformal class except in  the conformally flat case. The genuine fixed points  of the Cotton flow are the locally conformally flat metrics. However
there exist  other geometric fixed points, which correspond to Cotton solitons. A pseudo-Riemannian
manifold $(M,g)$ is a \emph{Cotton soliton} if it admits a vector field $\VE$ such that
\begin{equation}\label{Cotton equation}
    \mathcal{L}_\VE g+\Cottontensor=\lambda g,
\end{equation}
where $\mathcal{L}_\VE$ denotes the Lie derivative in the direction of the vector field $\VE$ and
$\lambda$ is a real number. A Cotton soliton is said to be \emph{shrinking, steady} or
\emph{expanding}, respectively, if $\lambda>0$, $\lambda = 0$ or $\lambda < 0$. Since there is no
ambiguity, we call Cotton soliton both to  the pseudo-Riemannian manifold $(M,g)$  and to the
vector field $\VE$. 

Cotton solitons are closely related to Ricci and Yamabe solitons,  which are defined by
$\mathcal{L}_\VE g+\rho=\lambda g$ and $\mathcal{L}_\VE g=(\tau-\lambda) g$, respectively. In particular, if
$(M,g)$ is a locally conformally flat homogeneous manifold then  the class of Cotton solitons
coincides with the class of Yamabe solitons (see for example
\cite{Calvino-Seoane-Vazquez-Vazquez}). In such a case  the Cotton soliton is said to be
 \emph{trivial}. On the other hand, if $(M,g)$ is a Lorentzian manifold which satisfies the field
equation of a Topologically Massive Gravity space then $(M,g)$ is a Ricci soliton if and only if it
is a  Cotton soliton \cite{Lashkari-Maloney10}.

In Riemannian signature, any  compact Cotton soliton is locally conformally flat, while compact
Lorentzian examples are available in the non locally conformally flat setting \cite{CL-GR-VL}.
Moreover, the fact that any   left-invariant homogeneous Ricci or Yamabe soliton on a
three-dimensional Riemannian Lie group is flat, while non-flat examples exist in the Lorentzian
signature (see \cite{Brozos-Calvaruso-Garcia-Gavino, Calvino-Seoane-Vazquez-Vazquez}), motivates a
study of homogeneous Cotton solitons in the Lorentzian setting.

Complete and simply connected three-dimensional  Lorentzian homogeneous manifolds are either
symmetric or a Lie group with a left-invariant Lorentzian metric \cite{Calvaruso07b}. Since
three-dimensional locally symmetric Lorentzian manifolds are locally conformally flat,  the purpose  of this paper is twofold. Firstly, to classify invariant homogeneous Cotton solitons on Lie
groups. Secondly, to determine algebraic Cotton solitons on Lie groups and  use them in order to obtain new examples of non-invariant Cotton solitons on homogeneous manifolds, both in the Riemannian and  Lorentzian settings.

%
%
%
%
%
%
%
%
%

\subsection*{Conventions and structure  of the paper.}
Throughout this paper, $(M,g)$ denotes a three-dimensional pseudo-Riemannian manifold and $(G,g)$
denotes  a three-dimensional Lie group equipped with a left-invariant metric; as usual,  $R$ stands
for the curvature tensor taken with the sign convention
$R(X,Y)=\nabla_{[X,Y]}-[\nabla_X,\nabla_Y]$, where $\nabla$ denotes the Levi-Civita connection. The
Ricci tensor, $\rho$, and the corresponding Ricci operator, $\widehat{\rho}$, are given by
$\rho(X,Y)=g( \widehat{\rho}(X),Y)$ $=$ $\operatorname{trace}\{Z \mapsto R(X,Z)Y\}$, and we denote
by $\tau$ the scalar curvature. Finally, $(G,g)$ is always assumed to be connected and simply
connected.

We organize this paper as follows. In Section \ref{se:preliminaries}   we review the description of
all three-dimensional   Lorentzian Lie algebras. We analyze the existence of non-trivial
left-invariant Cotton solitons on three-dimensional Lorentzian Lie groups in Section \ref{sect:left
invariant}.  In Section \ref{sect:algebraic Cotton} we determine algebraic Cotton solitons on homogeneous manifolds, showing the existence of Cotton solitons which are non-invariant. Finally,  in Section \ref{sect:non-invariant Lorentzian} examples of non-invariant shrinking and
expanding  Lorentzian homogeneous Cotton solitons which are non-trivial are constructed on the
Heisenberg group and on the  $E(1,1)$ group.

\section{Preliminaries}\label{se:preliminaries}

Let $\times$ denote the Lorentzian vector product on $\mathbb{R}^3_1$ induced by the product of the
para-quaternions (i.e., $e_1\times e_2=-e_3$, $e_2\times e_3=e_1$, $e_3\times e_1=e_2$, where $\{
e_1,e_2,e_3\}$ is an orthonormal basis of signature $(++-)$). Then $[Z,Y]=L(Z\times Y)$ defines a
Lie algebra, which is unimodular if and only if $L$ is a self-adjoint endomorphism of
$\mathfrak{g}$ \cite{Rahmani92}. Considering the different Jordan normal forms of $L$, we have the
following four classes of unimodular three-dimensional Lorentzian Lie algebras.

\medskip

\noindent{\underline{\emph{Type Ia}}.} If $L$ is diagonalizable with eigenvalues
$\{\alpha,\beta,\gamma\}$ with respect to an orthonormal basis $\{e_1,e_2,e_3\}$ of signature
$(++-)$, then the corresponding Lie algebra is given by

\medskip
\noindent
 $\noindent
(\mathfrak{g}_{Ia})\quad [e_1,e_2]=-\gamma e_3,\quad [e_1,e_3]=-\beta e_2,\quad [e_2,e_3]=\alpha
e_1. $

\bigskip

\noindent{\underline{\emph{Type Ib}}.} Assume $L$ has a complex eigenvalue. Then, with respect to
an orthonormal basis $\{e_1,e_2,e_3\}$ of signature $(++-)$, one has
\[
L=\left(\begin{array}{ccc}
\alpha &0&0\\
0&\gamma&-\beta\\
0 &\beta &\gamma
\end{array}\right), \qquad \beta\neq 0
\]
and thus the corresponding Lie algebra is given by

\medskip
\noindent
 $\noindent
\text{($\mathfrak{g}_{Ib}$)}\quad
 [e_1,e_2]=\beta e_2-\gamma e_3,\quad
[e_1,e_3]=-\gamma e_2-\beta e_3,\quad [e_2,e_3]=\alpha e_1. $

\bigskip

\noindent{\underline{\emph{Type II}}.} Assume $L$ has a double root of its minimal polynomial.
Then, with respect to an orthonormal basis $\{e_1,e_2,e_3\}$ of signature $(++-)$, one has
\[
L=\left(\begin{array}{ccc}
\alpha & 0&0\\
     0  &\frac{1}{2}+\beta&-\frac{1}{2}\\
     0  & \frac{1}{2}&-\frac{1}{2}+\beta
\end{array}\right)
\]
and thus the corresponding Lie algebra is given by

\medskip
\noindent
 $\noindent
(\mathfrak{g}_{II})\quad
 [e_1,e_2]=\frac{1}{2}e_2-(\beta-\frac{1}{2}) e_3,\quad
[e_1,e_3]=-(\beta+\frac{1}{2}) e_2-\frac{1}{2} e_3,\quad [e_2,e_3]=\alpha e_1. $

\bigskip

\noindent{\underline{\emph{Type III}}.} Assume $L$ has a triple root of its minimal polynomial.
Then, with respect to an orthonormal basis $\{e_1,e_2,e_3\}$ of signature $(++-)$, one has
\[
L=\left(\begin{array}{ccc}
\alpha &\frac{1}{\sqrt{2}} &\frac{1}{\sqrt{2}}\\
    \frac{1}{\sqrt{2}}   &\alpha&0\\
    -\frac{1}{\sqrt{2}}   & 0&\alpha
\end{array}\right)
\]
and thus the corresponding Lie algebra is given by

\medskip
\noindent
 $\noindent
(\mathfrak{g}_{III})\quad
 [e_1,e_2]=-\frac{1}{\sqrt{2}}e_1-\alpha e_3,\,
 [e_1,e_3]=-\frac{1}{\sqrt{2}} e_1-\alpha e_2,\,
 [e_2,e_3]=\alpha e_1+\frac{1}{\sqrt{2}}(e_2- e_3).
$

\bigskip

Next we treat the non-unimodular case. First of all, recall that a solvable Lie algebra
$\mathfrak{g}$ belongs to the special class $\algebraNomizu$ if $[x,y]$ is a linear combination of
$x$ and $y$  for any pair of elements in $\mathfrak{g}$. Any left-invariant metric on
$\algebraNomizu$ is of constant sectional curvature \cite{Milnor,nomizu} and hence locally conformally flat. Now, consider the
unimodular kernel, $\mathfrak{u}=\ker (\operatorname{trace}\, ad:\mathfrak{g}\rightarrow
\mathbb{R})$. It follows from \cite{Cordero-Parker97}  that non-unimodular Lorentzian Lie algebras
not belonging to class $\algebraNomizu$ are given, with respect to a suitable basis $\{ e_1,e_2,e_3\}$,
by

\medskip
\noindent
 $
(\mathfrak{g}_{IV})\quad
 [e_1,e_2]\!=\!0{,}\quad
[e_1,e_3]\!=\!\alpha e_1 + \beta e_2,\quad [e_2,e_3]\!=\! \gamma e_1 + \delta e_2{,} $ \quad
$\alpha+\delta\neq 0{,}$

\medskip
\noindent where one of the following holds:
\begin{enumerate}
\item[IV.1] $\{ e_1,e_2,e_3\}$ is orthonormal with  $g( e_1,e_1)=-g( e_2, e_2)=-g( e_3,e_3)=-1$
    and the structure constants satisfy $\alpha\gamma-\beta\delta=0$.
\item[IV.2] $\{ e_1,e_2,e_3\}$ is orthonormal with $g(e_1,e_1)=g( e_2, e_2)=-g( e_3,e_3)=1$ and
    the structure constants satisfy $\alpha\gamma+\beta\delta=0$.
\item[IV.3] $\{ e_1,e_2,e_3\}$ is pseudo-orthonormal with $g(e_1,e_1)=-g(e_2,e_3) = 1$ and the
    structure constants satisfy $\alpha\gamma=0$.
\end{enumerate}

\medskip

As a matter of notation, henceforth we will write
\[
   \nabla_{e_i}e_j = \sum_{k} \chr_{ij}^k e_k
\]
to represent the Levi-Civita connection corresponding to the left-invariant metric on the Lie
group, where $\{e_1, e_2, e_3\}$ denotes the basis fixed in  each case.  Moreover, we will denote
by  $\VE=\sum \ve_i e_i=(\ve_1,\ve_2,\ve_3)$  a generic vector field expressed in the corresponding
basis.

Three-dimensional locally conformally flat Lorentzian Lie groups have been studied by Calvaruso in
\cite{Calvaruso07}. We translate his classification to our context, in order to fit the notation
used throughout this paper.

\begin{lemma}\label{le: Lie groups with W=0}
A three-dimensional Lorentzian Lie group $(G,g)$ is locally conformally flat if and only if one of
the following  holds:
\begin{enumerate}
\item[(i)] $(G,g)$ is locally symmetric and
    \begin{itemize}
    \item[(i.a)] of Type Ia with $\alpha$ $=$ $\beta$  $=$ $\gamma$ or any cyclic
        permutation of $\alpha=\beta$, $\gamma=0$ (in any of these cases $(G,g)$ is of
        constant sectional curvature), or

    \item[(i.b)] of Type II with $\alpha=\beta=0$, and hence flat, or

    \item[(i.c)] of Type IV.1 with constant sectional curvature, or otherwise $\alpha$ $=$
        $\beta$ $=$ $\gamma$ $=$ $0$ and $\delta\neq 0$, or $\beta$ $=$ $\gamma$ $=$
        $\delta$ $=$ $0$ and $\alpha\neq 0$, or

    \item[(i.d)] of Type IV.2 with constant sectional curvature, or otherwise $\alpha$ $=$
        $\beta$ $=$ $\gamma$ $=$ $0$ and $\delta\neq 0$, or $\beta$ $=$ $\gamma$ $=$
        $\delta$ $=$ $0$ and $\alpha\neq 0$, or

    \item[(i.e)] of Type IV.3 and flat, or otherwise $\gamma$ $=$ $\delta$ $=$  $0$ and
        $\alpha\neq 0$,   or

    \item[(i.f)] of Type $\algebraNomizu$ and therefore of constant sectional curvature.
    \end{itemize}

\smallskip

\item[(ii)] $(G,g)$ is not locally symmetric and
    \begin{itemize}
    \item[(ii.a)] of Type Ib with $\alpha=-2\gamma$ and $\beta=\pm\sqrt{3}\,\gamma$, or
    \item[(ii.b)] of Type III with $\alpha=0$, or
    \item[(ii.c)] of Type IV.3 with $\gamma=0$ and $\alpha\delta(\alpha-\delta)\neq 0$.
    \end{itemize}
\end{enumerate}

\end{lemma}

\medskip

Finally, note that  any two vector fields $X_1$ and $X_2$ satisfying Equation \eqref{Cotton
equation} ($\mathcal{L}_{X_i}g+C=\lambda_ig$, $i=1,2$) differ in a homothetic vector field since
\[
\mathcal{L}_{X_1-X_2}g-(\lambda_1-\lambda_2)g =\mathcal{L}_{X_1}g-\lambda_1g -\mathcal{L}_{X_2}g+\lambda_2g
=0.
\]
Conversely,  adding a homothetic  vector field to any Cotton soliton gives another Cotton soliton.
As a consequence, if a homogeneous Lorentzian manifold admits two distinct Cotton solitons (i.e.,
for  constants $\lambda_1\neq \lambda_2$) then it is locally conformally flat and therefore trivial
(see \cite{Calvino-Seoane-Vazquez-Vazquez,Hall-Capocci99}  for more information on homogeneous manifolds admitting non-Killing homothetic vector fields).

\section{Left-invariant Cotton solitons on Lorentzian Lie groups }\label{sect:left invariant}

Now we consider the existence of left-invariant solutions of  Equation \eqref{Cotton equation} on
the Lie algebras discussed in Section \ref{se:preliminaries}. We completely solve the corresponding
equations, obtaining a complete description of all non-trivial left-invariant Cotton solitons.

Since the Cotton tensor is trace-free, if a Killing vector field $X$ satisfies Equation
\eqref{Cotton equation} then $C=0$ and $(M,g)$ is locally conformally flat. Conversely, if $(M,g)$
is locally conformally flat and homogeneous, then any Cotton soliton is a homothetic vector field
and hence it is Killing or otherwise the Ricci operator is two-step nilpotent and $(M,g,X)$ is also
a Yamabe soliton \cite{Calvino-Seoane-Vazquez-Vazquez}. In the homogeneous setting    a Cotton
soliton is said to be \emph{trivial} if $C=0$. In what follows  we will  focus on the non-trivial
case and therefore  we will use repeatedly  Lemma \ref{le: Lie groups with W=0} to exclude the case
of locally conformally flat Lie groups.

As a consequence of our analysis in this section the following geometric characterization of
Lorentzian Lie groups admitting  invariant Cotton solitons will be obtained.

\begin{theorem}\label{th:cotton-nilpotente}
A   Lorentzian Lie group $(G,g)$ admits a non-trivial left-invariant Cotton soliton if and only if
the Cotton operator is  nilpotent.
\end{theorem}

Previous result remarks the difference between the Riemannian and Lorentzian settings, since any nilpotent self-adjoint operator vanishes identically in the Riemannian category.

\subsection{Unimodular case}\label{sect:unimodular}
In this subsection we consider the existence of left-invariant Cotton solitons on three-dimensional
unimodular Lorentzian Lie groups whose corresponding Lie algebras were introduced in Section
\ref{se:preliminaries}.

\subsubsection{Type Ia}
In this first case the Levi-Civita connection is  determined by
\[
\begin{array}{l}
    \chr_{12}^3=\chr_{13}^2 = \frac{1}{2}(\alpha-\beta-\gamma),
    \\[0.05in]
    \chr_{21}^3 =\chr_{23}^1= \frac{1}{2}(\alpha-\beta+\gamma),
    \\[0.05in]
    \chr_{31}^2=-\chr_{32}^1= \frac{1}{2}(\alpha+\beta-\gamma),
\end{array}
\]
expressions which, for a left-invariant vector field $\VE=\sum \ve_i e_i$, allow  us to calculate
the non-zero terms in the Lie derivative of the metric, $\mathcal{L}_\VE g$,  given by
\begin{equation}\label{eq:Ia-Lie metric}
        (\mathcal{L}_\VE g)_{12}=(\alpha-\beta) \ve_3,
        \quad
        (\mathcal{L}_\VE g)_{13}=(\gamma-\alpha)\ve_2,
        \quad
        (\mathcal{L}_\VE g)_{23}= (\beta-\gamma)\ve_1,
\end{equation}
and also to determine  the non-zero components of the Cotton tensor, $\Cottontensor$,  as
\begin{equation}\label{eq:Ia-Cotton tensor}
\begin{array}{l}
    \Cottontensor_{11}=\phantom{-}
    \frac{1}{2} \left(2 \alpha ^3 - \alpha ^2 (\beta +\gamma )
    - (\beta -\gamma )^2 (\beta +\gamma )\right),
    \\[0.05in]
    \Cottontensor_{22}=\phantom{-}
    \frac{1}{2} \left(2 \beta ^3 - \beta ^2(\alpha+\gamma)
    - (\alpha-\gamma)^2(\alpha+\gamma)\right),
    \\[0.05in]
    \Cottontensor_{33}= -\frac{1}{2} \left(2 \gamma ^3-\gamma ^2 (\alpha +\beta )
    -(\alpha -\beta )^2 (\alpha +\beta )\right).
\end{array}
\end{equation}

\medskip

In this first case we show that left-invariant Cotton solitons are necessarily trivial.

\begin{theorem}\label{theor: Type Ia}
If a Type Ia unimodular Lorentzian Lie group is a left-invariant Cotton soliton then it is
necessarily trivial.
\end{theorem}

\begin{proof}
From Equations (\ref{eq:Ia-Lie metric}) and (\ref{eq:Ia-Cotton tensor}), a left-invariant vector
field $\VE=(\ve_1,\ve_2,\ve_3)$ is a Cotton soliton  if and only if
\begin{equation}\label{Soliton Ia}
\left\{
\begin{array}{l}
    \ve_1 (\beta -\gamma )
    = \ve_2 (\gamma -\alpha )
    = \ve_3 (\alpha -\beta )=0,
    \\[0.075in]
    \alpha ^2 (\beta +\gamma )+(\beta -\gamma )^2 (\beta +\gamma )-2 \alpha ^3+2 \lambda =0,
    \\[0.065in]
    \beta^2(\alpha+\gamma)+(\alpha-\gamma)^2(\alpha+\gamma)-2\beta^3+2\lambda = 0,
    \\[0.075in]
    \gamma ^2 (\alpha +\beta )+(\alpha -\beta )^2 (\alpha +\beta )-2 \gamma ^3+2\lambda=0.
\end{array}
\right.
\end{equation}

\noindent Note that the first equation in (\ref{Soliton Ia}) implies that the existence of non-zero
Cotton solitons is possible only if $\alpha=\beta$, or $\alpha=\gamma$, or $\beta=\gamma$.  Next
assume that $\alpha=\beta$ (the study of the other two cases is analogous). Under this condition,
again the first equation in (\ref{Soliton Ia}) implies that either $\beta=\gamma$ or
$\ve_1=\ve_2=0$.  In the first case, if $\beta=\gamma$,   the Lie group is of constant sectional
curvature $-\frac{\alpha^2}{4}$. In the second case, if $\ve_1=\ve_2=0$, Equation (\ref{Soliton
Ia}) reduces to
\[
\gamma^2(\beta-\gamma)-2 \lambda=0,\qquad
\gamma^2 (\beta-\gamma )+\lambda=0,
\]
and it follows that if $\beta\neq \gamma$ then necessarily $\gamma=0$ and  the Lie group is flat.
We conclude that, in any case, the Cotton soliton is trivial.
\end{proof}

\subsubsection{Type Ib} The Levi-Civita connection is determined by
\[
\begin{array}{l}
    \chr_{12}^3=\chr_{13}^2 = \frac{1}{2}(\alpha-2\gamma),
    \\[0.05in]
    \chr_{21}^2=-\chr_{22}^1=-\chr_{31}^3=-\chr_{33}^1 = -\beta,
    \\[0.05in]
    \chr_{21}^3=\chr_{23}^1=\chr_{31}^2=-\chr_{32}^1 = \frac{\alpha}{2},
\end{array}
\]
and a straightforward calculation shows that the Lie derivative of the metric is determined by
\begin{equation}\label{eq:Ib-Lie metric}
\begin{array}{l}
 (\mathcal{L}_\VE g)_{12} = \ve_2 \beta + \ve_3 (\alpha -\gamma ),
  \\[0.05in]
 (\mathcal{L}_\VE g)_{13} =  \ve_3\beta + \ve_2 (\gamma -\alpha ),
 \\[0.05in]
 (\mathcal{L}_\VE g)_{22} = (\mathcal{L}_\VE g)_{33} =  -2 \beta  \ve_1,
\end{array}
\end{equation}
while the Cotton tensor is characterized by
\begin{equation}\label{eq:Ib-Cotton tensor}
\begin{array}{l}
    \Cottontensor_{11}=-2\Cottontensor_{22}=2\Cottontensor_{33}=
    \alpha ^3  - \alpha ^2 \gamma + 4 \beta ^2 \gamma,
    \\[0.05in]
    \Cottontensor_{23}=
        \frac{1}{2} \beta  \left(\alpha ^2 +4 \beta ^2-8 \gamma ^2+4 \alpha  \gamma\right).
\end{array}
\end{equation}

\medskip

As in the Type Ia unimodular case, we do not obtain non-trivial left-invariant Cotton solitons in
this second case.

\begin{theorem}\label{theor:Type Ib}
A Type Ib unimodular Lorentzian Lie group does not admit any non-zero  left-invariant Cotton
soliton.
\end{theorem}

\begin{proof}
For a left-invariant vector field $\VE=(\ve_1,\ve_2,\ve_3)$, Equations (\ref{eq:Ib-Lie metric}) and
(\ref{eq:Ib-Cotton tensor}) imply that $\VE$ is a Cotton soliton if and only if
\begin{equation}\label{Soliton Ib}
\left\{
\begin{array}{l}
    \ve_2\beta +\ve_3 (\alpha -\gamma )=0,
    \\[0.075in]
    \ve_3\beta + \ve_2 (\gamma -\alpha ) =0,
    \\[0.075in]
    \beta  \left(\alpha ^2+4 \alpha  \gamma +4 \beta ^2-8 \gamma ^2\right)=0,
    \\[0.075in]
    \alpha ^3 - \alpha ^2 \gamma + 4 \beta ^2 \gamma - \lambda=0,
    \\[0.075in]
    \alpha ^3-\alpha ^2 \gamma +4 \beta ^2 \gamma +2 \lambda +4 \beta  \ve_1=0,
    \\[0.075in]
    \alpha ^3 - \alpha ^2 \gamma +4\beta^2\gamma +2\lambda-4\beta \ve_1=0.
\end{array}
\right.
\end{equation}

\noindent Note that $\beta\neq 0$. Hence, the first equation in (\ref{Soliton Ib}) implies that
$x_2=\frac{x_3 (\gamma-\alpha )}{\beta }$ and, as a consequence, the second equation in
(\ref{Soliton Ib}) is equivalent to $x_3 (\beta^2 + (\gamma - \alpha)^2)=0$. Therefore, it follows
that $x_3=0$ and hence $x_2=0$. Finally, from the last two equations in (\ref{Soliton Ib}) one
obtains $x_1=0$, and this ends the proof.
\end{proof}

\subsubsection{Type II}
In this case, the Levi-Civita connection is determined by
\[
\begin{array}{ll}
    \chr_{12}^3=\chr_{13}^2 = \frac{1}{2}(\alpha-2\beta),
    &
    \chr_{21}^2=-\chr_{22}^1=-\chr_{31}^3=-\chr_{33}^1= - \frac{1}{2},
    \\[0.05in]
    \chr_{21}^3 = \chr_{23}^1= \frac{1}{2}(\alpha-1),
    &
    \chr_{31}^2=-\chr_{32}^1= \frac{1}{2}(\alpha+1),
\end{array}
\]
and thus the Lie derivative of the metric is characterized by
\begin{equation}\label{eq:II-Lie metric}
    \begin{array}{l}
        (\mathcal{L}_\VE g)_{12}=\frac{1}{2} (\ve_2 + (2\alpha-2\beta-1)\ve_3),
        \\[0.05in]
        (\mathcal{L}_\VE g)_{13}=\frac{1}{2} (\ve_3 - (2\alpha-2\beta+1)\ve_2 )
        \\[0.05in]
        (\mathcal{L}_\VE g)_{22}=(\mathcal{L}_\VE g)_{33} = -(\mathcal{L}_\VE g)_{23}= -\ve_1.
    \end{array}
\end{equation}
Moreover, the Cotton tensor is determined by
\begin{equation}\label{eq:II-Cotton tensor}
\begin{array}{l}
    \Cottontensor_{11}=\alpha ^2 (\alpha -\beta),
    \\[0.05in]
    \Cottontensor_{22}=-\frac{1}{4} \left(2\alpha ^3 -8 \beta ^2 + 4 \alpha  \beta - \alpha ^2 (2 \beta-1) \right),
    \\[0.05in]
    \Cottontensor_{23}= \frac{1}{4}(\alpha ^2-8\beta ^2+4 \alpha  \beta),
    \\[0.05in]
    \Cottontensor_{33}=\frac{1}{4}
        \left(2 \alpha ^3+8 \beta ^2-4 \alpha  \beta  - \alpha ^2 (2 \beta+1)\right).
\end{array}
\end{equation}

\medskip

Next we determine the left-invariant Cotton solitons in the Type II unimodular case.

\begin{theorem}\label{theor soliton II}
A unimodular Lorentzian Lie group $(G,g)$ of Type II admits non-trivial left-invariant Cotton solitons if
and only if one of the following conditions holds:
\begin{itemize}
    \item[(i)] $\alpha=\beta\neq 0$, and then  $G=O(1,2)$ or $G=SL(2,\mathbb{R})$,
    \item[(ii)] $\alpha=0\neq \beta$,   and then $G=E(1,1)$.
\end{itemize}
Moreover, in these cases, the Cotton solitons are always steady and  respectively given by:
\begin{itemize}
    \item[(i)] $X=\frac{3}{4}\beta^2  e_1+\kappa (e_2+e_3)$, where $\kappa\in\mathbb{R}$.

\smallskip

    \item[(ii)] $X= 2 \beta ^2e_1$.
\end{itemize}
\end{theorem}

\begin{proof}
Considering Equations (\ref{eq:II-Lie metric}) and (\ref{eq:II-Cotton tensor}), the Cotton soliton
condition  for a left-invariant vector field $\VE=(\ve_1,\ve_2,\ve_3)$ can be expressed as
\begin{equation}\label{Soliton II}
\left\{
\begin{array}{l}
    \ve_2+\ve_3 (2 \alpha -2 \beta -1)=0,
    \\[0.075in]
    \ve_3-\ve_2 (2 \alpha-2 \beta +1)=0,
    \\[0.075in]
    \alpha ^2-8 \beta ^2+4 \alpha  \beta +4 \ve_1 =0,
    \\[0.075in]
    \alpha ^2 (\alpha-\beta )-\lambda=0,
    \\[0.075in]
     2 \alpha^3 - 8 \beta^2 + 4 \alpha \beta   - \alpha^2 (2 \beta-1) +4 \ve_1  +4 \lambda=0,
    \\[0.075in]
    2 \alpha ^3+8 \beta ^2  - 4 \alpha  \beta  - \alpha ^2 (2 \beta +1)- 4 \ve_1 +4 \lambda =0.
\end{array}
\right.
\end{equation}

\noindent  Adding the first and the second equations in (\ref{Soliton II}) we obtain $(\alpha
-\beta) (\ve_2-\ve_3)=0$. Hence, either $\alpha=\beta$ or $\ve_2=\ve_3$. Suppose first that
$\alpha=\beta$; in this case, the fourth equation in (\ref{Soliton II}) implies that $\lambda=0$.
Now, the first equation in (\ref{Soliton II}) leads to $\ve_2=\ve_3$ and finally the last equation
in (\ref{Soliton II}) implies that $\ve_1=\frac{3}{4}\beta^2$. With this set of conditions Equation
(\ref{Soliton II}) holds and the Cotton soliton is non-trivial whenever $\alpha=\beta\neq 0$, which
shows (i).

Next we analyze the case $\ve_2=\ve_3$, with $\alpha\neq \beta$. In this case, the first equation
in (\ref{Soliton II}) implies that $\ve_3(\alpha-\beta)=0$ and hence $\ve_3=0$. On the other hand,
from the fourth equation in (\ref{Soliton II}) we have $\lambda=\alpha^2(\alpha-\beta)$. Moreover,
from the last two equations in (\ref{Soliton II}) one easily obtains   $\alpha^2 (\alpha -\beta
)=0$, and therefore $\alpha=0$. Finally, Equation (\ref{Soliton II}) reduces to $\ve_1=2\beta^2$,
which shows (ii).
\end{proof}

\begin{remark}\rm
Left-invariant Cotton solitons on a unimodular Lorentzian Lie group of Type II are gradient if and
only if $\alpha=0\neq \beta$. To show this, we  analyze the two different cases obtained in Theorem \ref{theor soliton II}. First, assume
that $\alpha=\beta\neq 0$; in this case, left-invariant Cotton solitons are of the form
$X=\frac{3}{4}\beta^2 e_1+\kappa (e_2+e_3)$, $\kappa\in\mathbb{R}$, and hence the dual form $X^b$
is given by
\[
    X^b=\frac{3}{4}\beta^2 e^1+\kappa (e^2- e^3).
\]
A straightforward calculation shows that
\[
    dX^b=-\beta \kappa  (e^1\wedge e^2- e^1\wedge e^3)-{\textstyle \frac{3}{4}} \beta^3 e^2\wedge e^3
\]
and therefore the Cotton soliton is never gradient. Now suppose that $\alpha=0\neq \beta$; in this
second case, the  left-invariant Cotton soliton is given by $X=2\beta ^2e_1$. The dual form is
\[
    X^b=2\beta ^2e^1
\]
and, as a consequence, $dX^b=0$. Thus we conclude that there exists a smooth function $f$ such that
$X=\nabla f$.
\end{remark}

\subsubsection{Type III}
In this last unimodular case, the Levi-Civita connection is  determined by
\[
\begin{array}{l}
    \chr_{11}^2=-\chr_{11}^3=-\chr_{12}^1=-\chr_{13}^1
    = \chr_{22}^3=\chr_{23}^2=\chr_{32}^3=\chr_{33}^2 = \frac{1}{\sqrt{2}},
    \\[0.05in]
    \chr_{12}^3=\chr_{13}^2 =-\chr_{21}^3=-\chr_{23}^1=-\chr_{31}^2
    =\chr_{32}^1 = -\frac{\alpha}{2},
\end{array}
\]
and hence the Lie derivative of the metric is given  by
\begin{equation}\label{eq:III-Lie metric}
\begin{array}{l}
    (\mathcal{L}_\VE g)_{11} = -\sqrt{2} (\ve_2+\ve_3),
    \quad
    (\mathcal{L}_\VE g)_{12} = (\mathcal{L}_\VE g)_{13} =  \frac{\ve_1}{\sqrt{2}},
    \\[0.05in]
    (\mathcal{L}_\VE g)_{22} = \sqrt{2}\, \ve_3,
    \quad
    (\mathcal{L}_\VE g)_{23} = \frac{\ve_3-\ve_2}{\sqrt{2}},
    \quad
    (\mathcal{L}_\VE g)_{33} =  -\sqrt{2} \,\ve_2,
\end{array}
\end{equation}
while the Cotton tensor is characterized by
\begin{equation}\label{eq:III-Cotton tensor}
\begin{array}{l}
    \Cottontensor_{12}=\Cottontensor_{13}=\frac{3 \alpha^2}{2 \sqrt{2}},
    \qquad
    \Cottontensor_{22}=\Cottontensor_{23}=\Cottontensor_{33}=3\alpha.
\end{array}
\end{equation}

\medskip

For this case, we get the following.

\begin{theorem}\label{theor soliton III}
A unimodular Lorentzian Lie group $(G,g)$ of Type III  admits non-trivial left-invariant Cotton solitons if
and only if $\alpha\neq 0$, and then $G=O(1,2)$ or $G=SL(2,\mathbb{R})$. Moreover, in such a   case, the Cotton soliton is steady and given by
$X=-\frac{3 \alpha ^2}{2}e_1+\frac{3 \alpha }{\sqrt{2}} (e_2-e_3)$.
\end{theorem}

\begin{proof}
For a left-invariant vector field $\VE=(\ve_1,\ve_2,\ve_3)$, Equations (\ref{eq:III-Lie metric})
and (\ref{eq:III-Cotton tensor}) imply that $\VE$ is a Cotton soliton if and only if
\begin{equation}\label{Soliton III}
\left\{
\begin{array}{l}
    2 \ve_1+3 \alpha^2=0,
    \\[0.075in]
     \sqrt{2}\, \ve_2 + \sqrt{2}\, \ve_3 + \lambda=0,
    \\[0.075in]
    3 \sqrt{2}\, \alpha - x_2+ \ve_3=0,
    \\[0.075in]
    3 \alpha  + \sqrt{2}\, \ve_3 - \lambda =0,
    \\[0.075in]
    3 \alpha  -  \sqrt{2}\, \ve_2 + \lambda =0.
\end{array}
\right.
\end{equation}

\noindent If $\alpha=0$, then   it follows from Equation (\ref{Soliton III}) that the Lie group
does not admit any non-zero left-invariant  Cotton soliton. Next suppose that $\alpha\neq 0$. The
first equation in (\ref{Soliton III}) implies that $x_1=-\frac{3\alpha^2}{2}$ and the second
equation in (\ref{Soliton III}) implies that $x_2=-x_3 - \frac{\lambda}{\sqrt{2}}$. Now, from  the
last two equations in (\ref{Soliton III}) we get $\lambda=0$ and $x_3=-\frac{3\alpha}{\sqrt{2}}$,
from where the result follows.
\end{proof}

\begin{remark}\rm
A unimodular Lorentzian Lie group of Type III  does not admit  any left-invariant gradient Cotton
soliton. Indeed, considering the Cotton soliton determined in Theorem \ref{theor soliton III}, the dual form $X^b$
is given by
\[X^b=-\frac{3 \alpha ^2}{2}e^1+\frac{3 \alpha }{\sqrt{2}}(e^2+e^3)\]
and a straightforward calculation shows that
\[dX^b= \frac{3\alpha^2}{2\sqrt{2}}  \left( e^1\wedge e^2
                    + e^1\wedge e^3+ \sqrt{2}\,\alpha  e^2\wedge e^3\right).\]
Hence, the Cotton soliton $X$ is not gradient.
\end{remark}

\subsection{Non-unimodular case}\label{sect:non-unimodular}
In this subsection the existence of left-invariant Cotton solitons on three-dimensional
non-unimodular Lorentzian Lie groups is considered; the corresponding Lie algebras were introduced
in Section  \ref{se:preliminaries}. We exclude the study of  Type $\algebraNomizu$, since any Lie
group of that type is of constant sectional curvature and it does not admit  any non-zero Cotton
soliton. Next we show that in any other non-unimodular case     the  Cotton solitons are trivial.

\subsubsection{Type IV.1}

In this case, the Levi-Civita connection is  determined by
\[
\begin{array}{ll}
    \chr_{11}^3 =\chr_{13}^1=\alpha,
    &
    \chr_{12}^3=-\chr_{13}^2=\chr_{21}^3=\chr_{23}^1=-\frac{\beta-\gamma}{2},
    \\[0.05in]
    \chr_{22}^3=-\chr_{23}^{2}=-\delta,
    &
    \chr_{31}^2=\chr_{32}^1=-\frac{\beta+\gamma}{2},
\end{array}
\]
and thus the Lie derivative of the metric is characterized by
\begin{equation}\label{eq:IV.1-Lie metric}
    \begin{array}{lll}
        (\mathcal{L}_\VE g)_{11} = -2\alpha\ve_3,
        &
        (\mathcal{L}_\VE g)_{12}= (\beta-\gamma) \ve_3,
        &
        (\mathcal{L}_\VE g)_{13}= \alpha \ve_1 + \gamma\ve_2,
        \\[0.05in]
        (\mathcal{L}_\VE g)_{22} = 2\delta \ve_3,
        &
        (\mathcal{L}_\VE g)_{23} = -\beta\ve_1 - \delta \ve_2.
        &
    \end{array}
\end{equation}
Moreover, the Cotton tensor is determined by
\begin{equation}\label{eq:IV.1-Cotton tensor}
\begin{array}{l}
    \Cottontensor_{11}=-\frac{1}{2} \left(\beta ^3-2\gamma ^3-\alpha ^2 \beta +  \beta  \gamma ^2
        +2 \gamma\delta ^2 -\beta\delta ^2    \right),
    \\[0.05in]
    \Cottontensor_{12}= \delta  (\alpha  (\alpha-\delta )-\beta  (\beta -\gamma )),
    \\[0.05in]
    \Cottontensor_{22}=-\frac{1}{2} \left(2 \beta ^3 -\gamma^3  -2 \alpha ^2 \beta
            -\beta ^2 \gamma +\gamma  \delta ^2 + \alpha  \beta  \delta\right),
    \\[0.05in]
    \Cottontensor_{33}=-\frac{1}{2} (\beta +\gamma ) (\alpha +\beta -\gamma -\delta )
        (\alpha -\beta +\gamma -\delta ).
\end{array}
\end{equation}

We show the non-existence of non-trivial left-invariant Cotton solitons in this case in the
following result.

\begin{theorem}\label{theor: Soliton IV.1}
If a Type IV.1 non-unimodular Lorentzian Lie group is a left-invariant Cotton soliton then it is
necessarily trivial.
\end{theorem}

\begin{proof} Throughout the proof recall that $\alpha\gamma-\beta\delta=0$ and
$\alpha+\delta\neq 0$. Equations (\ref{eq:IV.1-Lie metric}) and (\ref{eq:IV.1-Cotton tensor}) imply
that a left-invariant vector field $\VE=(\ve_1,\ve_2,\ve_3)$ is a Cotton soliton if and only if
\begin{equation}\label{Soliton IV.1}
\left\{
\begin{array}{l}
    \alpha  \ve_1+\gamma  \ve_2=0,
    \\[0.075in]
    \beta  \ve_1+\delta  \ve_2=0,
    \\[0.075in]
    (\beta +\gamma ) (\alpha +\beta -\gamma -\delta ) (\alpha -\beta +\gamma -\delta )+2 \lambda=0,
    \\[0.075in]
   \delta  (\alpha  (\alpha-\delta )-\beta  (\beta -\gamma ))+  (\beta -\gamma ) \ve_3=0,
    \\[0.075in]
    \beta ^3-2 \gamma ^3-\alpha ^2 \beta +\beta  \gamma ^2 + 2\delta^2 \gamma-\beta\delta^2
   +4 \alpha  \ve_3-2 \lambda=0,
   \\[0.075in]
    2 \beta ^3 -\gamma ^3-2 \alpha ^2 \beta -\beta ^2 \gamma +\gamma  \delta^2 +\alpha  \beta  \delta
    -4 \delta  \ve_3+2 \lambda =0.
\end{array}
\right.
\end{equation}

\noindent We analyze first the case $\alpha=0$; in this case, $\beta=0$ and $\delta\neq 0$. The
second equation in (\ref{Soliton IV.1}) implies that  $\ve_2=0$. On the other hand,  the fourth
equation in (\ref{Soliton IV.1}) reduces to  $\gamma \ve_3=0$ and therefore either $\gamma=0$ or
$\ve_3=0$. If $\gamma=0$ then the Lie group is locally conformally flat and  hence the Cotton
soliton is trivial. If $\ve_3=0$ and $\gamma\neq 0$ then the last two equations in (\ref{Soliton
IV.1}) imply that $ (\gamma -\delta ) (\gamma +\delta )=0$, which leads to the constancy of the
sectional curvature  equals to $-\frac{\delta^2}{4}$;  therefore the Cotton soliton must be
trivial.

Assume  now that $\alpha\neq 0$; in this case, $\gamma=\frac{\beta \delta}{\alpha}$ and the  first
equation in (\ref{Soliton IV.1}) implies that  $\ve_1=-\frac{\beta \delta \ve_2}{\alpha ^2}$, while
the second equation in (\ref{Soliton IV.1}) reduces to
\begin{equation}\label{eq: Fifth soliton IV.1}
 \left(\alpha^2 - \beta^2\right) \delta \ve_2=0.
\end{equation}
If  $\alpha^2 - \beta^2=0$ then $\alpha=\varepsilon\beta$ and $\gamma=\varepsilon \delta$, where
$\varepsilon^2=1$; under these conditions the Lie group is of constant sectional curvature
$-\frac{1}{4}(\beta+\varepsilon\delta)^2$ and hence the Cotton soliton is trivial.

Next we consider the case $\delta=0$ in Equation (\ref{eq: Fifth soliton IV.1}) and therefore
$\gamma=0$. The fourth equation in (\ref{Soliton IV.1}) is equivalent to $\beta \ve_3=0$. Now, for
$\beta=0$  the Lie group is locally conformally flat and  hence the Cotton soliton is trivial. On
the other hand, if $\ve_3=0$ and $\beta\neq 0$ then Equation (\ref{Soliton IV.1}) reduces to
\[
    \beta^3-\alpha^2\beta-2\lambda = 0,\qquad  \beta^3-\alpha^2\beta+\lambda=0.
\]
Therefore $\alpha^2-\beta^2=0$, which shows that the Lie group is of constant sectional curvature
and  the Cotton soliton is again trivial.

Finally, we analyze the case $\ve_2=0$ in Equation (\ref{eq: Fifth soliton IV.1}); we assume
$\delta\neq 0$ and $\alpha^2-\beta^2\neq 0$ to avoid the previous cases. Note that, in this case,
$\ve_1=0$ and the fourth equation in (\ref{Soliton IV.1}) transforms into
\[
    (\alpha -\delta ) \left(\delta  \left(\alpha ^2-\beta ^2\right)
    +\beta  \ve_3\right)=0.
\]
If $\alpha=\delta$ then Equation (\ref{Soliton IV.1}) reduces to $\lambda=0$ and $\delta\ve_3=0$;
hence $\ve_3=0$ and the Cotton soliton is zero. Therefore, we assume $\alpha\neq \delta$ and we
have $\delta \left(\alpha ^2-\beta ^2\right)    +\beta  \ve_3=0$. Note that $\beta$ must be
non-zero since we are assuming $\delta(\alpha^2-\beta^2)\neq0$. Thus,
$\ve_3=-\frac{\delta(\alpha^2-\beta^2)}{\beta}$ and from the third equation in (\ref{Soliton IV.1})
we have
\[
    \lambda=-\frac{\beta  \left(\alpha^2-\beta ^2\right) (\alpha -\delta)^2 (\alpha
   +\delta )}{2 \alpha^3}
\]
and Equation (\ref{Soliton IV.1}) reduces to
\[
\left\{
\begin{array}{l}
   \alpha^2 \beta ^2-3 \beta ^2 \delta ^2+2 \alpha  \beta ^2 \delta +  4 \alpha ^4 =0,
   \\[0.075in]
    3 \alpha ^2 \beta ^2-\beta^2\delta ^2-2 \alpha  \beta ^2 \delta -  4 \alpha^2\delta^2=0.
   \end{array}
   \right.
\]
Now, it follows that $\left(\alpha ^2+\beta ^2\right) \left(\alpha ^2-\delta ^2\right)=0$, which is
a contradiction, and this ends the proof.
\end{proof}

\subsubsection{Type IV.2}
In this case, the Levi-Civita connection and the Lie derivative of the metric are determined by
\[
\begin{array}{ll}
    \chr_{11}^3 =\chr_{13}^1=\alpha,
    &
    \chr_{12}^3=\chr_{13}^2=\chr_{21}^3=\chr_{23}^1=\frac{\beta+\gamma}{2},
    \\[0.05in]
    \chr_{22}^3=\chr_{23}^{2}=\delta,
    &
    \chr_{31}^2=-\chr_{32}^1=-\frac{\beta-\gamma}{2},
\end{array}
\]
and
\begin{equation}\label{eq:IV.2-Lie metric}
    \begin{array}{lll}
        (\mathcal{L}_\VE g)_{11} = 2\alpha\ve_3,
        &
        (\mathcal{L}_\VE g)_{12}= (\beta+\gamma) \ve_3,
        &
        (\mathcal{L}_\VE g)_{13}= -\alpha \ve_1 - \gamma\ve_2,
        \\[0.05in]
        (\mathcal{L}_\VE g)_{22} = 2\delta \ve_3,
        &
        (\mathcal{L}_\VE g)_{23} = -\beta\ve_1 - \delta \ve_2,
        &
    \end{array}
\end{equation}
respectively. Now, the Cotton tensor is given by
\begin{equation}\label{eq:IV.2-Cotton tensor}
\begin{array}{l}
    \Cottontensor_{11}=\frac{1}{2} \left(\beta ^3+2 \gamma^3+\alpha^2 \beta
    +\beta\gamma^2+2 \gamma\delta^2+\beta\delta^2\right),
    \\[0.05in]
    \Cottontensor_{12}= \delta  \left(\alpha(\alpha-\delta) +\beta  (\beta +\gamma
   )\right),
   \\[0.05in]
    \Cottontensor_{22}=-\frac{1}{2} \left(2 \beta^3+\gamma ^3 + 2 \alpha ^2 \beta
    +\beta^2 \gamma + \gamma  \delta ^2  -\alpha  \beta  \delta \right),
    \\[0.05in]
    \Cottontensor_{33}=-\frac{1}{2} (\beta -\gamma ) \left((\alpha -\delta )^2+(\beta +\gamma)^2\right).
\end{array}
\end{equation}

\medskip

For this case we have the following.

\begin{theorem}\label{theor:Type IV.2}
If a Type IV.2 non-unimodular Lorentzian Lie group is a left-invariant Cotton soliton then it is
necessarily trivial.
\end{theorem}

\begin{proof}
Throughout the proof recall that $\alpha\gamma+\beta\delta=0$ and $\alpha+\delta\neq 0$. For a
left-invariant vector field $\VE=(\ve_1,\ve_2,\ve_3)$, Equations (\ref{eq:IV.2-Lie metric}) and
(\ref{eq:IV.2-Cotton tensor}) imply that $X$ is a Cotton soliton if and only if
\begin{equation}\label{Soliton IV.2}
\left\{
\begin{array}{l}
    \alpha  \ve_1+\gamma  \ve_2=0,
    \\[0.075in]
    \beta  \ve_1+\delta  \ve_2=0,
    \\[0.075in]
    (\beta -\gamma ) \left((\alpha -\delta )^2+(\beta +\gamma)^2\right)-2\lambda=0,
    \\[0.075in]
   \delta  \left(\alpha(\alpha-\delta)+\beta  (\beta +\gamma )\right)
   +(\beta +\gamma )   \ve_3 =0,
    \\[0.075in]
    \beta ^3  +  2 \gamma^3  +  \alpha ^2 \beta  +  \beta  \gamma^2
    + 2 \delta^2  \gamma  +  \beta \delta^2   +4  \alpha  \ve_3  -2 \lambda =0,
    \\[0.075in]
    2 \beta^3  +  \gamma^3  +  2 \alpha^2 \beta  +  \beta^2 \gamma + \gamma  \delta^2
    -\alpha  \beta  \delta   -4 \delta  \ve_3   +2 \lambda =0.
\end{array}
\right.
\end{equation}

\noindent Assume first that $\alpha=0$. In this case, $\beta=0$ and $\delta\neq 0$, and from the
third and fifth equations in (\ref{Soliton IV.2}) one obtains $\gamma
\left(\gamma^2+\delta^2\right)=0$. Therefore $\gamma=0$  and the Lie group is locally conformally
flat;   thus the left-invariant  Cotton soliton is trivial.

Next suppose that $\alpha\neq 0$. Hence $\gamma=-\frac{\beta \delta }{\alpha }$ and the first
equation in (\ref{Soliton IV.2}) implies that  $\ve_1=\frac{\beta \delta \ve_2}{\alpha^2}$. Now,
the second equation in (\ref{Soliton IV.2}) is equivalent to
\begin{equation}\label{IV.2-subcaso}
     \left(\alpha ^2+\beta ^2\right) \delta  x_2=0,
\end{equation}
and therefore either $\delta=0$ or $\ve_2=0$.  If $\delta=0$ then the fourth equation in
(\ref{Soliton IV.2}) reduces to $\beta \ve_3=0$. Note that if  $\ve_3=0$ then Equation
(\ref{Soliton IV.2}) reduces to
\[
    \beta^3+\alpha^2\beta+\lambda = 0,\qquad \beta^3+\alpha^2\beta -2\lambda=0,
\]
which implies that $\beta=0$. Hence, in any case, necessarily $\beta =0$, and it follows that the
Lie group is locally conformally flat. Thus the Cotton soliton is trivial.

Finally, we analyze the case $\ve_2=0$ in Equation (\ref{IV.2-subcaso}); we assume $\delta\neq 0$
to avoid the previous case. Note that, in this case, $\ve_1=0$ and  the fourth equation in
(\ref{Soliton IV.2}) is equivalent to
\[
    (\alpha - \delta) (\delta (\alpha^2 + \beta^2) +\beta\ve_3 )=0.
\]
If $\alpha=\delta$ then  Equation (\ref{Soliton IV.2}) reduces to $\lambda=0$ and $\delta\ve_3=0$;
hence $\ve_3=0$ and the Cotton soliton is zero. Now, if $\alpha\neq \delta$ then $\delta (\alpha^2
+ \beta^2) +\beta\ve_3=0$. Note that  $\beta$ must be non-zero since we are assuming that
$\delta\neq 0$. Thus, $\ve_3=-\frac{\delta \left(\alpha ^2+\beta ^2\right)}{\beta }$ and the third
equation in (\ref{Soliton IV.2}) implies that
\[
    \lambda=\frac{\beta  \left(\alpha ^2+\beta ^2\right)
    (\alpha -\delta )^2 (\alpha +\delta)}{2 \alpha ^3}.
\]
Now, Equation (\ref{Soliton IV.2}) reduces to
\[
\left\{
\begin{array}{l}
    4 \alpha ^4-\alpha ^2 \beta ^2-2 \alpha  \beta ^2 \delta +3 \beta^2 \delta ^2=0,
    \\[0.05in]
   \alpha ^2 \left(3 \beta ^2+4 \delta ^2\right)-2 \alpha  \beta ^2 \delta -\beta ^2 \delta ^2=0,
\end{array}
\right.
\]
and it follows that  $(\alpha -\beta ) (\alpha +\beta ) (\alpha -\delta ) (\alpha +\delta )=0$.
Since we are assuming that $\alpha\neq \delta$ and, moreover, $ \alpha+\delta\neq 0$, then we have
$\alpha=\pm\beta$ and in such a case the above system of equations has no solution (since
$\delta\neq 0$). This ends the proof.
\end{proof}

\subsubsection{Type IV.3}
In this case, the Levi-Civita connection is determined by
\[
\begin{array}{ll}
    \chr_{11}^2 = \chr_{13}^1= \alpha,
    &
    \chr_{12}^2=-\chr_{13}^3=\chr_{21}^2=\chr_{23}^1=
    -\chr_{31}^3=-\chr_{32}^1= \frac{\gamma}{2},
    \\[0.05in]
   \chr_{31}^2= \chr_{33}^1= -\beta,
    &
    \chr_{32}^2=-\chr_{33}^3= -\delta.
\end{array}
\]
Hence, the Lie derivative of the metric and the Cotton tensor are characterized by
\begin{equation}\label{eq:IV.3-Lie metric}
    \begin{array}{l}
        (\mathcal{L}_\VE g)_{11} = 2\alpha\ve_3,
        \qquad
        (\mathcal{L}_\VE g)_{12}= \gamma \ve_3,
        \qquad
        (\mathcal{L}_\VE g)_{13}= -\alpha \ve_1 - \gamma\ve_2-\beta\ve_3,
        \\[0.05in]
        (\mathcal{L}_\VE g)_{23} = -\delta \ve_3,
        \qquad
        (\mathcal{L}_\VE g)_{33} = 2(\beta \ve_1 + \delta \ve_2),
    \end{array}
\end{equation}
and
\begin{equation}\label{eq:IV.3-Cotton tensor}
\begin{array}{l}
    \Cottontensor_{11}=\gamma^3,
    \qquad
    \Cottontensor_{23}=\frac{\gamma^3}{2},
    \qquad
    \Cottontensor_{33}=\frac{\beta\gamma^2}{2}.
\end{array}
\end{equation}

\medskip

Next we show that left-invariant Cotton solitons reduce to left-invariant Yamabe solitons in this
case.

\begin{theorem}\label{Lemma soliton IV.3}
If a Type IV.3 non-unimodular Lorentzian Lie group is a left-invariant Cotton soliton then it is
necessarily trivial.
\end{theorem}

\begin{proof}
Using Equations (\ref{eq:IV.3-Lie metric}) and (\ref{eq:IV.3-Cotton tensor}) it follows that
a left-invariant vector field $\VE=(\ve_1,\ve_2,\ve_3)$ is a Cotton soliton if and only if
\begin{equation}\label{Soliton IV.3}
\left\{
\begin{array}{l}
    \gamma  \ve_3=0,
    \\[0.075in]
    \alpha \ve_1+\gamma  \ve_2+\beta  \ve_3=0,
    \\[0.075in]
    2 \alpha  \ve_3+\gamma ^3- \lambda =0,
    \\[0.075in]
    \gamma ^3 -2\delta  \ve_3 + 2\lambda=0,
    \\[0.075in]
    \beta  \gamma ^2 +4 \beta  \ve_1 + 4 \delta  \ve_2=0.
\end{array}
\right.
\end{equation}

\noindent Recall that $\alpha\gamma=0$ and $\alpha+\delta\neq 0$.  First, if $\alpha=\gamma=0$ then
the Lie group is   flat and hence the left-invariant Cotton soliton is trivial.  Now, if $\alpha=0$
and $\gamma\neq 0$ then the first equation in (\ref{Soliton IV.3}) implies that $\ve_3=0$, and
therefore from the third and fourth equations in (\ref{Soliton IV.3}) we easily obtain that
$\gamma=0$, which is a contradiction. Finally, if $\gamma=0$ and $\alpha\neq 0$ then  the Lie group
is locally conformally flat and therefore the left-invariant Cotton soliton is trivial.
\end{proof}

\begin{remark}\rm
In the Riemannian setting, the unimodular Lie groups correspond to Type Ia (just considering the
usual cross product induced by the quaternions), while the non-unimodular case corresponds to Type
$\algebraNomizu$ and Type IV.2 previously discussed \cite{Milnor}.

With regard to Types Ia and IV.2, the  behaviour  is exactly the same in both the Riemannian and the
Lorentzian cases, since Theorems \ref{theor: Type Ia} and \ref{theor:Type IV.2} remain true in the
Riemannian setting. Thus, since locally symmetric spaces and Lie groups of Type $\algebraNomizu$
are locally conformally flat, one has that three-dimensional homogeneous Riemannian manifolds do
not admit any non-trivial left-invariant Cotton soliton. Hence, any left-invariant Riemannian
Cotton soliton is a Yamabe soliton, and in Riemannian signature this implies the flatness of the
manifold.

Let $H_3$ be the Heisenberg group and consider on $H_3$  the left-invariant metric given by $g=
dx^2+dy^2+(dz-x dy)^2$.  Proceeding as in Section \ref{sect:non-invariant Lorentzian} it is
obtained that  $(H_3,g)$ is a shrinking non-invariant Riemannian Cotton soliton which is not
trivial.
\end{remark}

\bigskip
\noindent\emph{Proof of Theorem \ref{th:cotton-nilpotente}.}  A careful examination of the  cases obtained in Theorem \ref{theor soliton II} and in Theorem \ref{theor soliton III} shows that the
corresponding Cotton operator is two-step nilpotent for unimodular Lorentzian Lie groups of Type II
with $\alpha=\beta\neq 0$ or $\alpha=0\neq \beta$, while the degree of nilpotency is three for
unimodular Lorentzian Lie groups of Type III with $\alpha\neq 0$. Conversely, a case by case
examination of the Cotton operator in the different unimodular and non-unimodular Lorentzian Lie
groups shows, after a long but straightforward calculation, that the  cases with nilpotent Cotton
operator are precisely those obtained in   Theorems \ref{theor soliton II} and  \ref{theor soliton III}.$\hfill\square$

\section{Algebraic Lorentzian Cotton solitons}\label{sect:algebraic Cotton}

Cotton solitons are fixed points of the Cotton flow up to diffeomorphisms and rescaling. The specific  behaviour of a homogenous manifold allows   us to consider a stronger condition than the Cotton soliton one. More precisely, we can consider soliton solutions for the Cotton flow up to automorphisms instead of diffeomorphisms; in such a case the soliton is referred to as an algebraic Cotton soliton. Let $(G,g)$ be a simply connected Lie group equipped with a left-invariant Lorentzian metric $g$ and let $\mathfrak{g}$ denote the Lie algebra of the Lie group $G$. Following the seminal work of Lauret \cite{Lauret}, $(G,g)$ is said to be  an \emph{algebraic Cotton soliton} if it satisfies
\begin{equation}\label{eq:algebraic cotton soliton}
\widehat{C}=\lambda\operatorname{Id}+D
\end{equation}
where $\widehat{C}$ stands for the Cotton operator ($g(\widehat{C}(X),Y)=C(X,Y)$), $\lambda$ is a real constant and $D\in\operatorname{Der}(\mathfrak{g})$, i.e.,
\[
D[X,Y]=[DX,Y]+[X,DY]\quad \text{for all} \, X, Y \in\mathfrak{g}.
\]

Next we show that the  algebraic Cotton soliton condition is stronger than the Cotton soliton one.

\begin{proposition}
Let $(G,g)$ be a simply connected Lie group endowed with a left-invariant Lorentzian metric $g$. If $(G,g)$ is an algebraic Cotton soliton, i.e. it satisfies Equation (\ref{eq:algebraic cotton soliton}), then it is a Cotton soliton such that a vector field solving Equation (\ref{Cotton equation}) is given by
\[
X=\dfrac{d}{dt}|_{t=0}\varphi_t(p),\quad with \quad d\varphi_t|_e=\operatorname{exp}\left(\dfrac{t}{2}D\right),
\]
where $e$ denotes the identity element of $G$.
\end{proposition}
\begin{proof}
Let $(G,g)$ be an algebraic Cotton soliton, i.e.  $\widehat{C}=\lambda\operatorname{Id}+D$, let  $\{e_i\}$ denote  a pseudo-orthonormal basis and define a one-parameter family of automorphisms $\varphi_t$  by setting  $d\varphi_t{|_e}=\operatorname{exp}\left(\dfrac{t}{2}D\right)$. Considering  the vector field $X$ given by $X=\dfrac{d}{dt}|_{t=0}\varphi_t(p)$ a direct calculation  shows that the Lie derivative in the direction of  $X$ is given by
\[
\left(\mathcal{L}_X g\right)(e_i,e_j)=\dfrac{d}{dt}|_{t=0}\varphi^*_t g(e_i,e_j)=\dfrac{1}{2}\left(g(De_i,e_j)+g(e_i,De_j)\right),
\]
for  $i,j\in\{1,2,3\}$. Then,
\[
\begin{array}{rcl}
C(e_i,e_j) & = & \frac{1}{2}\left(g(\widehat{C}(e_i),e_j)+g(e_i,\widehat{C}(e_j))\right)\\
\noalign{\bigskip}
\phantom{C(e_i,e_j)} & = & \frac{1}{2}\left(g((\lambda\operatorname{Id}+D)(e_i),e_j)+g(e_i,(\lambda\operatorname{Id}+D)(e_j))\right)\\
\noalign{\bigskip}
\phantom{C(e_i,e_j)} & = & \lambda\, g(e_i,e_j)+\left(\mathcal{L}_X g\right)(e_i,e_j),
\end{array}
\]
which shows that $(G,g,X)$ is a Cotton soliton
\end{proof}

In view of the  proposition above it seems natural to ask whether a Cotton soliton on a Lie group comes from an algebraic Cotton soliton or not. The remaining of this  section is devoted to clarify this question. In this sense, the next theorem shows how algebraic Cotton solitons allow constructing non-invariant Cotton solitons.

\begin{theorem}\label{th:algebraic Cotton solitons}
A three-dimensional Lie group $(G,g)$  equipped with a left-invariant Lorentzian metric is an  algebraic Cotton soliton if and only if one of the following conditions holds:
\begin{itemize}
\item[(i)] $(G,g)$ is of Type Ia with $\alpha=\beta=0$ and $\gamma\neq 0$, or any cyclic permutation. In this case,  $\lambda = -2 \gamma^3$ and $G=H_3$ is  the Heisenberg group.
\item[(ii)] $(G,g)$ is of Type Ib with $\alpha=0$ and $\gamma=\dfrac{\varepsilon \sqrt{2}\beta}{2}$. In this case, $\lambda=2\varepsilon  \sqrt{2} \beta^3$ with $\varepsilon^2=1$, and  $G=E(1,1)$.
\end{itemize}
\end{theorem}
\begin{proof}
We analyze the existence of algebraic Cotton solitons on a three-dimensional Lorentzian Lie algebra. To do that,  it is sufficient to study when an operator $D$ of the form
\[D=\widehat{C}-\lambda\operatorname{Id}\]
is a derivation on a Lie algebra $\mathfrak{g}$.

We start with  Type Ia. From Equation (\ref{eq:Ia-Cotton tensor}) it is easy to get that  $D=\widehat{C}-\lambda\operatorname{Id}$ is a derivation if and only if
\begin{equation}\label{eq: algebraic cotton Ia}
\left\{
\begin{array}{lll}
\gamma  \left(\alpha ^3-\alpha ^2 \beta -\alpha  \left(\beta ^2-\gamma ^2\right)+\beta ^3+\beta  \gamma ^2-2 \gamma ^3-\lambda \right)&=&0,\\
\noalign{\medskip}
\beta  \left(\alpha ^3-\alpha ^2 \gamma -\alpha  \left(\gamma ^2-\beta ^2\right)-2 \beta ^3+\beta ^2 \gamma +\gamma ^3-\lambda \right)&=&0,\\
\noalign{\medskip}
\alpha  \left(2 \alpha ^3-\alpha ^2 (\beta +\gamma )-\beta ^3+\beta ^2 \gamma +\beta  \gamma ^2-\gamma ^3+\lambda \right)&=&0.
\end{array}
\right.
\end{equation}
If all the  structure constants vanish then the Lie algebra is abelian and hence the algebraic Cotton soliton is trivial. Hence, we  assume that at least one of them is non-zero, for instance $\gamma\neq 0$. Thus, from the first equation in (\ref{eq: algebraic cotton Ia}) we obtain that $\lambda=\gamma ^2 (\alpha +\beta )+(\alpha -\beta )^2 (\alpha +\beta )-2 \gamma ^3$. Then, Equation (\ref{eq: algebraic cotton Ia}) becomes
\begin{equation}\label{eq: algebraic cotton Ia reduced}
\left\{
\begin{array}{lll}
\beta  (\beta -\gamma ) \left(\alpha ^2+2 \alpha  (\beta +\gamma )-3 \beta ^2-2
   \beta  \gamma -3 \gamma ^2\right)&=&0,\\
\noalign{\medskip}
\alpha  (\alpha -\gamma ) \left(3 \alpha ^2+2 \alpha  (\gamma -\beta )-(\beta
   -\gamma ) (\beta +3 \gamma )\right)&=&0.
\end{array}
\right.
\end{equation}
A first non-trivial solution is obtained taking $\alpha=\beta=0$ and $\lambda=-2\gamma^3$. For $\alpha\neq 0$ and $\beta=0$ a straightforward calculation  shows that $\alpha=\gamma$, and in this case the manifold is locally conformally flat. The same conclusion is obtained for $\alpha=0$ and $\beta\neq 0$. Finally, assume $\alpha\neq 0\neq \beta$. If $\alpha=\gamma$ or $\beta=\gamma$ or $\alpha=\beta$ Equation (\ref{eq: algebraic cotton Ia reduced}) implies that $\alpha=\beta=\gamma$ and therefore the sectional curvature is constant; thus the manifold is locally conformally flat. If $\alpha\neq\beta\neq\gamma\neq\alpha$ the Equation (\ref{eq: algebraic cotton Ia reduced}) reduces to 
\begin{equation}\label{eq: algebraic cotton Ia very reduced}
\left\{
\begin{array}{lll}
\alpha ^2+2 \alpha  (\beta +\gamma )-3 \beta ^2-2
   \beta  \gamma -3 \gamma ^2&=&0,\\
\noalign{\medskip}
3 \alpha ^2+2 \alpha  (\gamma -\beta )-(\beta
   -\gamma ) (\beta +3 \gamma )&=&0.
\end{array}
\right.
\end{equation}
From this last equation we get $(\alpha -\beta )^2+3 \gamma ^2=0$, which  is a contradiction. Proceeding in an analogous way for $\alpha\neq 0$ or $\beta\neq 0$ (i) is obtained. \\

We consider now the Type Ib. From Equation (\ref{eq:Ib-Cotton tensor}), and taking into account that $\beta\neq0$, it follows that  $D=\widehat{C}-\lambda\operatorname{Id}$ is a derivation if and only if
\begin{equation}\label{eq: algebraic cotton Ib}
\left\{
\begin{array}{rll}
\alpha ^3+4 \alpha  \gamma ^2+8 \beta ^2 \gamma -8 \gamma ^3-\lambda&=&0,\\
\noalign{\medskip}
\alpha ^2+4 \alpha  \gamma +4 \beta ^2-8 \gamma ^2&=&0,\\
\noalign{\medskip}
\alpha  \lambda&=&0.
\end{array}
\right.
\end{equation}
The last equation in (\ref{eq: algebraic cotton Ib}) implies that either  $\alpha=0$ or $\lambda=0$. First, suppose that  $\alpha=0$. Thus, Equation  (\ref{eq: algebraic cotton Ib}) becomes
\begin{equation}\label{eq: algebraic cotton Ib alpha zero}
\left\{
\begin{array}{rll}
8 \beta ^2 \gamma -8 \gamma ^3-\lambda&=&0,\\
\noalign{\medskip}
\beta ^2-2 \gamma ^2&=&0,\\
\end{array}
\right.
\end{equation}
from where (ii) is obtained. We assume now that $\lambda=0$ and $\alpha\neq 0$. Then, Equation (\ref{eq: algebraic cotton Ib}) becomes
\begin{equation}\label{eq: algebraic cotton Ib lambda zero}
\left\{
\begin{array}{rll}
8 \gamma ^3-\alpha ^3-4 \alpha  \gamma ^2-8 \beta ^2 \gamma &=&0,\\
\noalign{\medskip}
\alpha ^2+4 \alpha  \gamma +4 \beta ^2-8 \gamma ^2&=&0.\\
\end{array}
\right.
\end{equation}
A long but straightforward calculation  shows that
$\alpha=-2\gamma$ and $\beta=\pm\sqrt{3} \gamma$, and in this case the  manifold is locally conformally flat.

The remaining types of three-dimensional Lorentzian Lie algebras do  not provide non-trivial algebraic Cotton solitons. The proof is obtained by proceeding as in the previous cases. We omit the details for sake of brevity since the calculations are standard.
\end{proof}

\begin{remark}\rm
Case (i) in Theorem \ref{th:algebraic Cotton solitons} is the Lie algebra associated to the Heisenberg group, while case (ii) corresponds to the Lie algebra associated to the solvable Lie group $E(1,1)$. An analogous treatment can be done in the Riemannian case to show that  only the Riemmanian Heisenberg group admits algebraic Cotton solitons. In this sense, in \cite{CL-GR-VL} the authors obtained the Cotton soliton associated to the Riemannian algebraic Cotton soliton on the Heisenberg group.
\end{remark}

\begin{remark}
\rm Theorem \ref{th:algebraic Cotton solitons} highlights that  left-invariant Cotton solitons do not come from algebraic Cotton solitons.
\end{remark}

\section{Non-invariant Lorentzian Cotton solitons}\label{sect:non-invariant Lorentzian}
In this section, motivated by the existence of algebraic Cotton solitons, we show \emph{the existence
of non-invariant Cotton solitons on three-dimensional homogeneous Lorentzian manifolds which are
non-trivial}. To do this, we consider the algebraic Cotton solitons obtained in Theorem \ref{th:algebraic Cotton solitons}. Note that  Theorem \ref{th:cotton-nilpotente} implies  that any Lie group admitting a left-invariant Cotton soliton has necessarily  nilpotent Cotton operator and this condition does not hold for the Lie groups obtained in Theorem \ref{th:algebraic Cotton solitons}. Thus, the Lie groups admitting algebraic Cotton solitons provide non-invariant examples of Cotton solitons.

 \subsection{Non-invariant Cotton solitons on the Heisenberg group}
We start with the Lorentzian Heisenberg group. In
\cite{Rahmani-Rahmani06} it is shown that this Lie group  can be endowed with three different
left-invariant Lorentzian metrics, up to isometry and scaling, given by
\[
\begin{array}{l}
    g_1=-dx^2+dy^2+(x dy +dz)^2,
    \\[0.05in]
    g_2=dx^2+dy^2-(x dy +dz)^2,
    \\[0.05in]
    g_3=dx^2+(x dy+dz)^2-((1-x)dy-dz)^2.
\end{array}
\]
Metric  $g_3$ is flat \cite{nomizu} and hence a trivial Cotton soliton, and  metrics $g_1$ and
$g_2$ are shrinking non-gradient Ricci solitons \cite{Onda10,Onda}. Next we analyze metrics $g_1$
and $g_2$ by separate. Starting with the metric $g_1$, for a vector field
$X=\mathcal{A}(x,y,z)\partial_x+\mathcal{B}(x,y,z)\partial_y+\mathcal{C}(x,y,z)\partial_z$ the Lie derivative  of the metric is
given by
\begin{equation}\label{eq:Lie-g1}
\begin{array}{l}
    (\mathcal{L}_X g_1)(\partial_x,\partial_x) = -2 \mathcal{A}_x,
    \\[0.05in]
    (\mathcal{L}_X g_1)(\partial_x,\partial_y) = -\mathcal{A}_y + \left(x^2+1\right) \mathcal{B}_x+x \mathcal{C}_x,
    \\[0.05in]
    (\mathcal{L}_X g_1)(\partial_x,\partial_z) = -\mathcal{A}_z+x \mathcal{B}_x+\mathcal{C}_x,
    \\[0.05in]
    (\mathcal{L}_X g_1)(\partial_y,\partial_y) =  2 \left(x \mathcal{A} + \left(x^2+1\right) \mathcal{B}_y+x \mathcal{C}_y\right),
    \\[0.05in]
    (\mathcal{L}_X g_1)(\partial_y,\partial_z) =  \mathcal{A}+x \mathcal{B}_y + \left(x^2+1\right) \mathcal{B}_z+\mathcal{C}_y+x  \mathcal{C}_z,
    \\[0.05in]
    (\mathcal{L}_X g_1)(\partial_z,\partial_z) =  2 \left(x \mathcal{B}_z+\mathcal{C}_z\right),
\end{array}
\end{equation}
and the $(0,2)$-Cotton tensor is determined by
\begin{equation}\label{eq:Cotton-g1}
\begin{array}{l}
    C(\partial_x,\partial_x)= -\frac{1}{2},
    \,\,
    C(\partial_y,\partial_y)= \frac{1}{2}-x^2,
    \,\,
    C(\partial_y,\partial_z)=  -x,
    \,\,
    C(\partial_z,\partial_z)= -1.
\end{array}
\end{equation}
Now, Equations (\ref{eq:Lie-g1}) and (\ref{eq:Cotton-g1}) imply that $X$ is a Cotton soliton if and
only if
\begin{equation}\label{eq:soliton-g1}
\left\{
\begin{array}{l}
    4 \mathcal{A}_x + 1 - 2\lambda  = 0,
    \\[0.05in]
    2  x \mathcal{B}_z  +  2 \mathcal{C}_z - 1 - \lambda = 0,
    \\[0.05in]
    \mathcal{A}_z - x\mathcal{B}_x - \mathcal{C}_x = 0,
    \\[0.05in]
    \mathcal{A}_y - (x^2+1) \mathcal{B}_x - x\mathcal{C}_x = 0,
    \\[0.05in]
    \mathcal{A}+x \mathcal{B}_y + \left(x^2+1\right) \mathcal{B}_z+\mathcal{C}_y+x  \mathcal{C}_z - x - x\lambda = 0,
    \\[0.05in]
    4x\mathcal{A} + 4(x^2+1) \mathcal{B}_y + 4x\mathcal{C}_y - 2x^2 +  1  - 2 (x^2+1)\lambda  = 0.
\end{array}
\right.
\end{equation}
We start integrating    first and second  equations in (\ref{eq:soliton-g1}) to get
\[
\begin{array}{l}
    \mathcal{A}(x,y,z) = \mathcal{A}_1(y,z) + \frac{1}{4}(2\lambda-1) x,
    \\[0.05in]
    \mathcal{C}(x,y,z) = \mathcal{C}_1(x,y) -  \mathcal{B}(x,y,z) x + \frac{1}{2}  (\lambda+1) z.
\end{array}
\]
Now, the third equation in (\ref{eq:soliton-g1}) transforms into
\[
    \mathcal{B}(x,y,z) = (\mathcal{C}_1)_x(x,y) - (\mathcal{A}_1)_z(y,z)
\]
and hence differentiating  the fourth equation in (\ref{eq:soliton-g1}) with respect  to $x$ and
$z$ we obtain $(\mathcal{A}_1)_{zz} = 0$, and thus
\[
    \mathcal{A}_1(y,z) = \mathcal{A}_2(y) z  + \mathcal{A}_3(y).
\]
Next, differentiating the fifth equation in (\ref{eq:soliton-g1}) with respect to $z$ we get
$\mathcal{A}_2(y)=0$ and the fourth equation in (\ref{eq:soliton-g1}) transforms into $(\mathcal{C}_1)_{xx}(x,y) -
\mathcal{A}_3'(y)=0$, from where it follows that
\[
\begin{array}{l}
    \mathcal{C}_1(x,y) = \frac{1}{2} \mathcal{A}_3'(y) x^2 + \mathcal{C}_2(y) x + \mathcal{C}_3(y).
\end{array}
\]
At this point, the fifth equation in (\ref{eq:soliton-g1}) reduces to
\[
    4\mathcal{A}_3(y) + 4\mathcal{C}_3'(y)+  (2 \mathcal{A}_3''(y) x + 4 \mathcal{C}_2'(y) - 3) x = 0
\]
and from this equation a straightforward calculation shows that
\[
\begin{array}{l}
    \mathcal{A}_3(y) = \kappa_1 y+\kappa_2,
    \quad
    \mathcal{C}_2(y) = \frac{3}{4} y+\kappa_3,
    \quad
    \mathcal{C}_3(y) = -\frac{\kappa_1}{2} y^2-\kappa_2 y+\kappa_4,
\end{array}
\]
where $\kappa_i\in\mathbb{R}$, $i=1,\dots,4$. Thus, Equation (\ref{eq:soliton-g1}) finally reduces
to $\lambda-2=0$ and  we   conclude that   $(H_3,g_1)$ is a Cotton soliton if and only if it is
shrinking with $\lambda=2$ and $X$ given by
\[
\begin{array}{l}
    X=(\kappa_1 y+\frac{3}{4} x+\kappa_2)\partial_x
        +( \kappa_1 x+\frac{3 }{4} y+\kappa_3)\partial_y
    -(\frac{\kappa_1}{2}(x^2+y^2)+\kappa_2y-\frac{3}{2} z-\kappa_4)\partial_z.
\end{array}
\]
Note that as an special case of the Cotton solitons obtained above, one has that $X=\frac{3}{4} x
\partial_x
        + \frac{3 }{4} y \partial_y + \frac{3}{2} z\partial_z$ defines a complete Cotton soliton on
        $(H_3,g_1)$.

Now, it is easy to see that the associated Cotton operator, $\widehat{C}$, is characterized by
\[
\begin{array}{l}
    \widehat{C}(\partial_x) =\frac{1}{2}\partial_x,
    \quad
    \widehat{C}(\partial_y) =\frac{1}{2}\partial_y - \frac{3}{2} x \partial_z,
    \quad
    \widehat{C}(\partial_z) =-\partial_z.
\end{array}
\]
As a consequence,  it diagonalizes with eigenvalues $\{-1,\frac{1}{2},\frac{1}{2}\}$ and hence the
Cotton soliton cannot be left-invariant (see Theorem \ref{th:cotton-nilpotente}). Moreover, the
Ricci operator, $\widehat{\rho}$, is given by
\[
\begin{array}{l}
    \widehat{\rho}(\partial_x) =\frac{1}{2}\partial_x,
    \quad
    \widehat{\rho}(\partial_y) =\frac{1}{2}\partial_y -   x \partial_z,
    \quad
    \widehat{\rho}(\partial_z) =-\frac{1}{2} \partial_z,
\end{array}
\]
and thus it does not have any  zero eigenvalue, which implies that  the  soliton is not gradient
\cite{CL-GR-VL}.

\begin{remark}\rm
A completely analogous study can be developed with the metric $g_2$ to obtain that $(H_3, g_2)$ is
a Cotton soliton if and only if it is expanding with $\lambda=-2$ and the soliton given by
\[
\begin{array}{l}
    X=(\kappa_1 y-\frac{3}{4} x +\kappa_2)\partial_x
        -( \kappa_1 x+\frac{3}{4}y -\kappa_3)\partial_y
    +(\frac{\kappa_1}{2}(x^2- y^2)-\kappa_2y-\frac{3}{2}z+\kappa_4)\partial_z,
\end{array}
\]
where $\kappa_i\in\mathbb{R}$, $i=1,\dots,4$.  Again the Cotton soliton is not left-invariant and
it is not  gradient.
\end{remark}

\subsection{Non-invariant Cotton solitons on $E(1,1)$}
The Lie group $E(1,1)$ admits two kinds of algebraic Cotton solitons (up to some sign $\varepsilon=\pm 1$, see Theorem \ref{th:algebraic Cotton solitons}) and therefore it  can be equipped with a  left-invariant metric with non-nilpotent Cotton operator admitting some Cotton soliton. Moreover, these Cotton solitons will be necessarily  non-invariant (see  Theorem \ref{th:cotton-nilpotente}).

For our purpose we consider the frame \cite{Onda}
\[
E_1=\partial_x,\quad E_2=\dfrac{1}{2}\left(e^x\partial_y+e^{-x}\partial_z\right),\quad E_3=\dfrac{1}{2}\left(e^x\partial_y-e^{-x}\partial_z\right)
\]
and the associated coframe given by
\[
E^1=dx,\quad E^2=e^{-x}dy+e^xdz,\quad E^3=e^{-x}dy-e^xdz.
\]
Henceforth  we consider left-invariant metrics corresponding to  Theorem \ref{th:algebraic Cotton solitons}-(ii). First, we take the metric $g_1$ given by
\[g_1=\frac{2}{3}(E^1)^2-(E^2)^2+2 \sqrt{6} (E^2E^3)-3 (E^3)^2,
\]
or in local coordinates
\[
g_1=\dfrac{2}{3}dx^2  + (2 \sqrt{6}-4)e^{-2 x} dy^2  -
 (2 \sqrt{6}+4) e^{2 x} dz^2 + 4 dy dz.
\]
Now, for a  vector field $X=\mathcal{A}(x,y,z)\partial_x+\mathcal{B}(x,y,z)\partial_y+\mathcal{C}(x,y,z)\partial_z$ the Lie derivative of the metric is given by 
\begin{equation}\label{eq:E11-g1}
\begin{array}{l}
    (\mathcal{L}_X g_1)(\partial_x,\partial_x) = \frac{4}{3} \mathcal{A}_x,
    \\[0.05in]
    (\mathcal{L}_X g_1)(\partial_x,\partial_y) =\frac{2}{3} \left(\mathcal{A}_y+3 \left(\sqrt{6}-2\right)
   e^{-2 x}
   \mathcal{B}_x+ 3 \mathcal{C}_x\right),
    \\[0.05in]
    (\mathcal{L}_X g_1)(\partial_x,\partial_z) = \frac{2}{3} \left(\mathcal{A}_z+3 \mathcal{B}_x-3
   \left(\sqrt{6} + 2\right) e^{2 x} \mathcal{C}_x\right),
    \\[0.05in]
    (\mathcal{L}_X g_1)(\partial_y,\partial_y) =  4 \mathcal{C}_y-4 \left(\sqrt{6}-2\right) e^{-2 x}
   \left(\mathcal{A}-\mathcal{B}_y\right),
    \\[0.05in]
    (\mathcal{L}_X g_1)(\partial_y,\partial_z) =  2 \left(\left(\sqrt{6}-2\right) e^{-2 x}
   \mathcal{B}_z+\mathcal{B}_y+\mathcal{C}_z-\left(\sqrt{6}+2\right) e^{2 x} \mathcal{C}_y\right),
    \\[0.05in]
    (\mathcal{L}_X g_1)(\partial_z,\partial_z) =  4 \mathcal{B}_z-4 \left(\sqrt{6}+2\right) e^{2 x}
   \left(\mathcal{A}+\mathcal{C}_z\right)
\end{array}
\end{equation}
and  a straightforward calculation shows that $X=\frac{3 y}{\sqrt{2}}\partial_y+\frac{3 z}{\sqrt{2}}\partial_z$ is a shrinking Cotton soliton with $\lambda=2\sqrt{2}$. Moreover, due to the non-existence of homothetic vector fields on $E(1,1)$ it follows that  this is the unique shrinking Cotton soliton up to Killing vector fields \cite{Calvino-Seoane-Vazquez-Vazquez}.

\medskip

Next we consider the Lorentzian metric $g_2$ on $E(1,1)$  given by
\[g_2=\frac{2}{3}(E^1)^2-3(E^2)^2+2 \sqrt{6} (E^2E^3)- (E^3)^2,
\]
or in local coordinates
\[
g_2=\dfrac{2}{3}dx^2 + ( 2 \sqrt{6} - 4) e^{-2 x} dy^2
- ( 2 \sqrt{6}+4) e^{2 x} dz^2  - 4 dy dz.
\]
Proceeding as in the previous case we get that  $X=-\frac{3y}{\sqrt{2}}\partial_y-\frac{3z}{\sqrt{2}}\partial_z$ is a expanding Cotton soliton for $\lambda=-2\sqrt{2}$. Furthermore, it is the unique expanding Cotton soliton up to Killing vector fields due to the results in \cite{Calvino-Seoane-Vazquez-Vazquez}.

\begin{remark}\rm
As a final remark, it is worth pointing out once again  that the non-invariant Cotton solitons constructed in this section come from algebraic Cotton solitons obtained in Theorem \ref{th:algebraic Cotton solitons}. Moreover, let us emphasize that algebraic Cotton solitons have an opposite behaviour with respect to invariant Cotton solitons on Lie groups since no  left-invariant Cotton soliton can be obtained from an algebraic one.
\end{remark}

\end{document}